\documentclass{amsart}
\pdfoutput=1 
\usepackage{amssymb}
\usepackage{tikz}

\newcommand\N{\mathbb N}

\newcommand\R{\mathbb R}

\newcommand\ph\varphi
\newcommand\ps\psi
\newcommand\ep\varepsilon
\newcommand\rh\varrho
\newcommand\al\alpha
\newcommand\be\beta
\newcommand\ga\gamma
\newcommand\om\omega
\newcommand\ta\tau
\renewcommand\th\theta
\newcommand\de\delta
\newcommand\ze\zeta
\newcommand\ch\chi
\newcommand\et\eta
\newcommand\io\iota
\newcommand\la\lambda
\newcommand\si\sigma

\newcommand\Ga\Gamma
\newcommand\De\Delta
\newcommand\Th\Theta
\newcommand\La\Lambda
\newcommand\Si\Sigma
\newcommand\Ph\Phi
\newcommand\Ps\Psi
\newcommand\Om\Omega

\newtheorem{theorem}{Theorem}

\theoremstyle{definition}

\newtheorem{example}[theorem]{Example}
\theoremstyle{remark}
\newtheorem{remark}[theorem]{Remark}

\begin{document}
\title[Non-commutative polytopes and polyhedra]
{A note on non-commutative polytopes and polyhedra}

\author{Beatrix Huber}
\author{Tim Netzer}
\address{University of Innsbruck, Department of Mathematics, Innsbruck, Austria}

\thanks{Supported by the Austrian Science Fund FWF through project P 29496-N35}

\begin{abstract} It is well-known that every polyhedral cone is finitely generated (i.e. polytopal), and vice versa. Surprisingly, the two notions differ almost always for non-commutative versions of such cones. This was obtained as a byproduct in \cite{fnt} and later generalized in \cite{PSS}. In this note we give a direct and constructive proof of the statement. Our proof yields a new and surprising quantitative result: the difference of the two notions can always be seen at the first level of non-commutativity, i.e. for matrices of size $2$, independent of dimension and  complexity of the initial convex cone. This also  answers an open question from \cite{PSS}.
\end{abstract}

\maketitle

\section{Introduction and Preliminaries}
A convex cone $C\subseteq\R^d$ is called  {\it polyhedral}, if there exist linear functionals $\ell_1,\ldots,\ell_m\colon\R^d\to\R$ with $$C=\left\{ a\in\R^d\mid \ell_1(a)\geq 0, \ldots, \ell_m(a)\geq 0\right\}.$$  A convex cone $C$ is called {\it finitely generated} (or {\it polytopal}) if there are $v_1,\ldots, v_n\in\R^d$ with $$C={\rm cc}\left\{ v_1,\ldots, v_n\right\}:=\left\{ \sum_{i=1}^n \lambda_iv_i\mid \lambda_1,\ldots, \lambda_n \geq 0 \right\}.$$ The Minkowski-Weyl-Theorem (see for example \cite{sch}) states that each polyhedral cone is finitely generated, and each finitely generated cone is polyhedral.

A recent development in real algebraic geometry and convexity theory is to consider {\it non-commutative} sets and cones. They arise by replacing points from $\R^d$ with $d$-tuples of Hermitian matrices (of arbitrary size). A lot of meaningful information about polynomials and semialgebraic sets comes to light when these non-commutative levels are added to the classical setup. Examples are Helton's Positiv-stellensatz \cite{h} and  the analysis of Ben-Tal and Nemirovski's algorithm for checking inclusion of spectrahedra \cite{bn,fnt,hkm}, among others (see also \cite{hkm2} for an overview). 
For a polyhedral/polytopal cone $$C=\left\{ a\in\R^d\mid \ell_1(a)\geq 0, \ldots, \ell_m(a)\geq 0\right\}={\rm cc}\left\{ v_1,\ldots, v_n\right\}$$ there are two natural ways to extend the cone to matrix levels. The first one uses the polyhedral description, and is the standard way of defining non-commutative semialgebraic sets by polynomial inequalities. For each $s\in\N$  we define $$C_s^{\rm ph} :=\left\{ (A_1,\ldots,A_d)\in {\rm Her}_s^d\mid \ell_i(A_1,\ldots, A_d)\geqslant 0, i=1,\ldots, m\right\},$$ where ${\rm Her}_s$ is the real vector space of complex Hermitian $s\times s$-matrices, and $\geqslant 0$ means that a matrix is positive semidefinite. Note that the above definition makes sense, since a real linear polynomial can be evaluated at a tuple of Hermitian matrices, and the result is again Hermitian. Also note that $C_1^{\rm ph}$ coincides with $C$. We now consider the collection over all matrix-sizes as our {\it non-commutative polyhedral extension of $C$}:
  $$C^{\rm ph} := \left( C_s^{\rm ph}\right)_{s\in\N}.$$ 
  The second non-commutative extension of $C$ uses the generators of $C$ and looks a little less natural at first sight. However, there are good reasons for the following definition, as we will argue below. We just replace nonnegative numbers by positive semidefinite matrices and define for any $s\in \N:$ $$C_s^{\rm pt}:=\left\{ \sum_{i=1}^n P_i\otimes v_i\mid P_i\in {\rm Her}_s, P_i\geqslant 0, i=1,\ldots, n \right\}.$$ Here, $\otimes$ denotes the Kronecker (=tensor) product of matrices. In our case it just means we put $P_i$ into each component of the vector $v_i$ and multiply it with the real number in there. The result is a $d$-tuple of Hermitian matrices of size $s$, and so is the sum over all $i$. Also note that $C_1^{\rm pt}$ again coincides with $C$, since positive semidefinite matrices of size $1$ are just nonnegative real numbers. Now the collection $$C^{\rm pt}:=\left( C_s^{\rm pt}\right)_{s\in\N}$$ is the {\it non-commutative polytopal extension of $C$}. 
  
  We will restrict to {\it proper convex cones} from now on, i.e. closed convex cones $C$ with nonempty interior and $C\cap -C=\{0\}.$ Then all $C_s^{\rm ph}$ and $C_s^{\rm pt}$ have the same property, and they fit well into the context of  {\it operator systems} (see \cite{fnt} and the references therein for details). In fact both $C^{\rm ph}$ and $C^{\rm pt}$are  abstract operator systems with $C$ at scalar level, and in particular convex in the non-commutative sense. It is easily seen (and shown in \cite{fnt}) that $C^{\rm ph}$ is the largest operator system with $C$ at scalar level, and $C^{\rm pt}$ is the smallest such operator system. In particular we have $C_s^{\rm pt}\subseteq C_s^{\rm ph}$ for all $s$ (which can also be easily  checked directly).

 \section{Main Result} 
Theorem \ref{thm:main} below is the main result of these notes. Without the information about the matrix size $2$, the result is a byproduct of the main results of \cite{fnt} (see Remark 4.9 from that work). However, the proof there is quite involved and non-constructive, in particular since the focus is on a different property of operator systems.  See also Remark \ref{rem:just} below for more comments on the difference of the two proofs.
The  main result was later generalized in Theorem 4.1\ from \cite{PSS}, to include the case that  $C$ is not polyhedral. It is also shown there that the difference of the cones can always be seen at level $2^{d-1}$.  In Problem 4.3\ the authors then ask whether this bound can be improved.  We now give  direct, simple and completely constructive  proof of the main result. It also answers Problem 4.3 in proving the somewhat surprising result about the matrix size $2$. 
  
 \begin{theorem}\label{thm:main} Let $C\subseteq\R^d$ be a proper polyhedral cone. 

(i)  If $C$ is a simplex cone, then $C^{\rm pt}=C^{\rm ph}$. 

(ii) If $C$ is not a simplex cone, then $C_2^{\rm pt}\subsetneq C_2^{\rm ph}.$
 \end{theorem} 
 \begin{proof}
Statement ($i$) is easy. The argument is the same as in \cite{fnt}, we repeat it for completeness. If $C$ is a simplex cone, then up to a linear isomorphism of the underlying space $\R^d$ we can assume $C=\R_{\geq 0}^d,$ the positive orthant. In that case one readily checks $$C_s^{\rm ph}=\left\{(A_1,\ldots, A_d)\in{\rm Her}_s^d\mid A_i\geqslant 0, i=1,\ldots, d\right\}=C_s^{\rm pt}$$ for all $s\in\N$. 

For ($ii$) assume that $C$ is not a simplex cone. We first settle the case of smallest possible dimension, namely $d=3$. Since $C$ is proper and has at least $4$ extremal rays, after applying a linear isomorphism we can assume that $C$ is generated by $$v_1=(1,-1,1), v_2=(-1,-1,1), v_3=(-1,1,1), v_4=(1,1,1)$$ and some $v_5,\ldots, v_n\in (1,\infty) \times (-1,1) \times\{1\}$ (see for example \cite{go} Section 2.8.1.\ for an explicit construction of such an isomorphism and Figure \ref{fig:poly} for the intersection of the  cone $C$ with the plane defined by $x_3=1$).

\begin{center}
\begin{figure}[h!]
\begin{tikzpicture}[scale=1.5]
\draw[->](-1.3,0)--(3.5,0);\put (155,0) {$x_1$};
\draw[->](0,-1.3)--(0,1.3);\put (0,60) {$x_2$};
\draw[red,thick]  (1,1)--(4.8,-1);
\draw[red,thick]  (-1,-1)--(4.8,-1);
\draw[red,thick]  (1,1)--(-1,1)--(-1,-1)--(1,-1);
\filldraw[red, opacity=0.3]  (1,1) -- (-1,1) -- (-1,-1) -- (1,-1) -- (4.8,-1)-- cycle;
\draw[thick,blue,opacity=0.6]
 (1,1) -- (-1,1) -- (-1,-1) -- (1,-1) -- (1.3,-0.8) -- (1.8,-0.1) -- (1.7,0.4) -- (1.45,0.7)-- cycle;
\filldraw[blue, opacity=0.4]
 (1,1) -- (-1,1) -- (-1,-1) -- (1,-1) -- (1.3,-0.8) -- (1.8,-0.1) -- (1.7,0.4) -- (1.45,0.7)-- cycle;
\filldraw (4.8,-1) circle (0.8pt);\put (200,-50) {$w$};
\filldraw (1,-1) circle (0.8pt); \put (40,-50) {$v_1$};
\filldraw (-1,-1) circle (0.8pt); \put (-45,-50) {$v_2$};
\filldraw (-1,1) circle (0.8pt); \put (-45,47) {$v_3$};
\filldraw (1,1) circle (0.8pt); \put (40,47) {$v_4$};
\filldraw  (1.45,0.7) circle (0.8pt); \put (65,33) {$v_5$};
\filldraw (1.7,0.4) circle (0.8pt);  \put (76,15) {$v_6$};
\filldraw (1.8,-0.1)  circle (0.8pt);  \put (83,-5) {$\hdots$};
\filldraw (1.3,-0.8) circle (0.8pt); \put (60,-35) {$v_n$};
\end{tikzpicture}
\caption{section with plane $x_3=1$ of $C$ (blue) and  $D$ (red)}\label{fig:poly}
\end{figure}
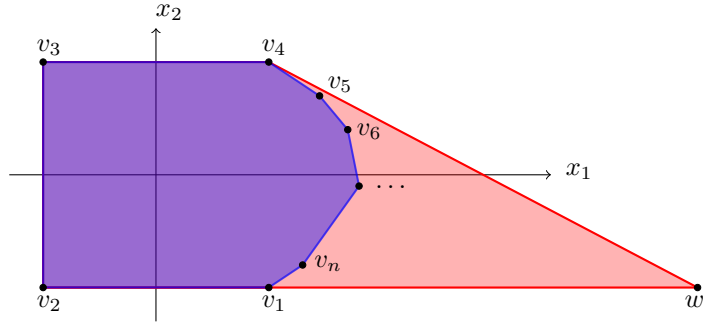
\end{center}
 
\vspace{-0.5cm}
 We now consider the matrix tuple $$\underline A:=(A_1, A_2,A_3):=\left(\left(\begin{array}{cc}1 & 0 \\0 & -1\end{array}\right), \left(\begin{array}{cc}0 & 1 \\1 & 0\end{array}\right),\left(\begin{array}{cc}1 & 0 \\0 & 1\end{array}\right)\right)\in{\rm Her}_2^3$$ and claim that $\underline A\in C_2^{\rm ph}.$ It is easily checked that $\underline A$ even fulfills  $$A_3\pm A_1\geqslant 0, \quad A_3\pm A_2\geqslant 0,$$ and by Farkas Lemma \cite{far} in particular the inequalities defining $C$.
 
 Let us prove  $\underline A\notin C_2^{\rm pt}.$ First choose another point $w=(\lambda,-1,1)$ with $\lambda$ so large that $$C\subseteq {\rm cc}\{w,v_2,v_3,v_4\}=:D$$ (see Figure \ref{fig:poly}). We now even prove $\underline A\notin D_2^{\rm pt}.$ Assume to the contrary that  there exists positive semidefinite matrices $P_1,P_2,P_3,P_4\in {\rm Her}_2$ with \begin{align*}\underline A&=(A_1,A_2,A_3)\\ &=P_1\otimes w+P_2\otimes v_2+P_3\otimes v_3+P_4\otimes v_4\\ &=(\lambda P_1-P_2-P_3+P_4,-P_1-P_2+P_3+P_4,P_1+P_2+P_3+P_4).\end{align*}
 Adding the first and third entry we obtain \begin{equation}\label{first}\left(\begin{array}{cc}2 & 0 \\0 & 0\end{array}\right)=A_1+A_3=(1+\lambda)P_1+2P_4,\end{equation} which implies $$P_1=\left(\begin{array}{cc}\alpha_1 & 0 \\0 & 0\end{array}\right), \quad P_4=\left(\begin{array}{cc}\alpha_4 & 0 \\0 & 0\end{array}\right) $$ for some $\alpha_1,\alpha_4\geq 0$, since $P_1,P_4\geqslant 0.$  Similarly we get $$\left(\begin{array}{cc}1 & -1 \\-1 & 1\end{array}\right)=A_3-A_2=2(P_1+P_2),$$ implying $$P_2=\frac{1}{2}\left(\begin{array}{cc}\alpha_2 & -1 \\-1 & 1\end{array}\right).$$  Plugging all of this into the equation for $A_1$ we  get $$\left(\begin{array}{cc}1 & 0 \\0 & -1\end{array}\right)=\left(\begin{array}{cc}\lambda\alpha_1-\alpha_2/2+\alpha_4 & 1/2 \\1/2 & -1/2\end{array}\right) -P_3,$$ implying  $$P_3=\left(\begin{array}{cc}\alpha_3 & 1/2 \\1/2 & 1/2\end{array}\right).$$ From $P_2,P_3\geqslant 0$ we obtain  $\alpha_2\geq 1$ and $\alpha_3\geq 1/2.$ From $I_2=P_1+P_2+P_3+P_4$ we thus find $\alpha_1=\alpha_4=0$, so $P_1=P_4=0$,  which contradicts ($\ref{first}$). This proves $\underline A\notin D_2^{\rm pt}\supseteq C_2^{\rm pt}$, and thus settles the case $d=3$.
 
 We now proceed by induction on $d$. Let $C\subseteq\R^d$ be a proper polyhedral cone which is not a simplex cone. Then either $C$ has a facet which is not a simplex cone, or a vertex figure which is not a simplex cone \cite{zi}. 

 In the first case we can assume that $C$ is contained in the halfspace defined by $x_1\leq 0$ and the non-simplex facet $F$ lies in the hyperplane defined by $x_1=0.$ We can apply the induction hypothesis to $F\subseteq \R^{d-1}$ and find $(A_2,\ldots,A_d)\in  F_2^{\rm ph}\setminus F_2^{\rm pt}.$ Then for $\underline A:=(0,A_2,\ldots, A_d)$ we obviously have $\underline A\in C_2^{\rm ph}.$ Now assume $\underline A\in C_2^{\rm pt}$. By looking at the first component in a representation $$(0,A_2,\ldots, A_d)=\sum_i P_i \otimes v_i$$ with $v_i\in C$ we see that $P_i\neq 0$ can only occur for $v_i\in F$. Indeed any $v_i\in C\setminus F$ has a negative first entry, and such terms cannot cancel to yield $0$. So the representation is a representation of $(A_2,\ldots, A_d)$ in $F_2^{\rm pt}$, which does not exist. So we have shown $\underline A\notin C_2^{\rm pt}.$

 In the second case we can assume that the non-simplex vertex-figure $F$ of $C$ is cut out by the hyperplane defined by $ x_1=0$, and further that $v_1$ spans the only extreme ray of $C$ with negative $x_1$-entry, whereas all other generators have a positive first entry (see Figure \ref{fig:verfig} for an illustration).
 \begin{center}
 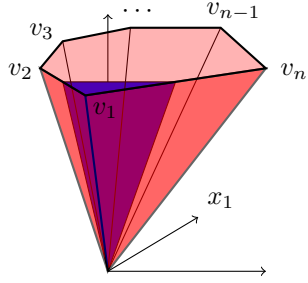
\begin{figure}[h!!]
 \begin{tikzpicture}[xscale=3,yscale=1.8]
    \draw[->] (0.5,0.5) -- (1.2,0.5);
     \draw[->] (0.5,0.5) -- (0.5,2.4);
     \draw[->] (0.5,0.5) -- (0.9,0.9) node[above right] {$x_1$};
      \filldraw[fill=blue] (0.5,0.5)--(0.3,1.9)--(0.8,1.9)--(0.5,0.5);
      \draw[]      (0.3,2.2)  -- (0.5,0.5);
      \draw[]    (0.6,2.3)-- (0.5,0.5);     
      \draw[]   (1,2.3)-- (0.5,0.5);   
      \draw[thick, fill=red, opacity=0.6] (0.5,0.5) -- (0.2,2)--(0.4,1.8)--(0.5,0.5) --(1.2,2)--(0.4,1.8) ;
       \draw[thick, fill opacity=0.3, fill=red]  (0.2,2) -- (0.3,2.2)  -- (0.6,2.3) --(1,2.3)--(1.2,2)-- (0.4,1.8)-- cycle;
 	\put (37,85) {$v_1$}
	 \put (5,100) {$v_2$}
	  \put (13,115) {$v_3$}
	 \put (48,122) {$\cdots$}
	 \put (80,123) {$v_{n-1}$}
	  \put (108,100) {$v_{n}$}
\end{tikzpicture}\caption{vertex figure $F$ (blue) of $C$ (red)}\label{fig:verfig}
\end{figure}
\end{center}
 After scaling the generators $v_i$ we can even assume that the $x_1$-component of $v_1$ is $-1$, and the $x_1$-component of all other $v_i$ is $1$.    Then the cone $F$ is generated by vectors $w_2,\ldots, w_n,$ where each $w_i$ is of the form $$w_i= \frac12v_1+\frac12v_i.$$  Since $F$ is not a simplex cone we can apply  the induction hypothesis  to $F\subseteq \R^{d-1}$ and  again find $(A_2,\ldots, A_d)\in F_2^{\rm ph}\setminus F_2^{\rm pt}.$ As before we now argue that $$\underline A:=(0,A_2,\ldots,A_d)\in C_2^{\rm ph}\setminus C_2^{\rm pt},$$ where  $\underline A \in C_2^{\rm ph}$ is again clear.  So assume for contradiction that $\underline A\in C_2^{\rm pt}$, so  there exists some positive semidefinite $P_i\in {\rm Her}_2$ with  $$(0,A_2,\ldots, A_d)=P_1\otimes v_1+ P_2\otimes v_2+\cdots +P_n\otimes v_n.$$  Since the first entry of this matrix tuple is zero, we get $P_1=P_2+\cdots +P_n$, which implies \begin{align*}\underline A&=  \left(P_2+\cdots +P_n\right)\otimes v_1+ P_2\otimes v_2+\cdots +P_n\otimes v_n\\ &= P_2\otimes (v_1+v_2) +\cdots +P_n\otimes (v_1+v_n) \\ &= 2P_2\otimes w_2+\cdots +2P_n\otimes w_n.\end{align*}
 This contradicts $(A_2,\ldots, A_n)\notin F_2^{\rm pt},$ and finishes the proof.
  \end{proof} 

 \begin{remark}\label{rem:just}
 (i) Let us comment on the difference of the above proof and the proof from \cite{fnt}. First, the main result from \cite{fnt} states that the abstract operator system $C^{\rm ph}$ admits a finite-dimensional realization, whereas $C^{\rm pt}$ does not. This of course implies that they cannot coincide, but gives no result on the level at which the differ. The proof starts in a similar fashion as the above, first settling the case $d=3$. But already here our construction of $\underline A$ is much more explicit and simpler than what was done in \cite{fnt}. The induction step in \cite{fnt} is completely non-constructive and cannot be transformed into an explicit argument. Our argument above is completely constructive. After applying the necessary isomorphisms and induction steps one obtains some explicit $\underline A\in C_2^{\rm ph}\setminus C_2^{\rm pt}.$
 
 (ii) Note that all appearing matrices above are real symmetric. So the difference between the cones appears not only in the Hermitian case, but already when we restrict ourselves to real symmetric matrices.
 \end{remark}

\begin{example} 
We consider the $3$-dimensional square-cone \begin{align*}C&=\left\{ a\in\R^3\mid a_3\pm a_1\geq 0, a_3\pm a_2\geq 0\right\}\\& ={\rm cc}\left\{(1,-1,1),(-1,-1,1),(-1,1,1),(1,1,1) \right\}.\end{align*} We have seen in the proof of Theorem \ref{thm:main} that    $$\underline A=\left(\left(\begin{array}{cc}1 & 0 \\0 & -1\end{array}\right), \left(\begin{array}{cc}0 & 1 \\1 & 0\end{array}\right),\left(\begin{array}{cc}1 & 0 \\0 & 1\end{array}\right)\right)\in C_2^{\rm ph}\setminus C_2^{\rm pt}.$$ So we can see the difference of the two cones for example in the affine subspace $$V:=\left\{ \left( \left(\begin{array}{cc}x & 0 \\0 & -1\end{array}\right),\left(\begin{array}{cc}0 & y \\y & 0\end{array}\right),\left(\begin{array}{cc}1 & 0 \\0 & 1\end{array}\right)\right)\mid x,y\in \R\right\}\subseteq {\rm Her}_s^3.$$ After identifying $V$ with $\R^2$ it is a straightforward computation to see that $$C_2^{\rm ph}\cap V=[-1,1]\times[-1,1].$$ Determining $C_2^{\rm pt}\cap V$ needs some more computation. After imposing all necessary linear constraints on $P_1,P_2,P_3,P_4\in {\rm Her}_2$ to ensure $$P_1\otimes v_1 +P_2\otimes v_2+ P_3\otimes v_3+ P_4\otimes v_4\in V,$$ then using the conditions that all  $P_i$ must be positive semidefinite, and then solving for $x$ and $y$, one gets $$C_2^{\rm pt}\cap V= \left\{ (x,y)\in [-1,1]^2\mid x+2y^2\leq 1\right\}.$$ Figure \ref{fig:sec} shows the two affine sections. The black dot corresponds to the point $\underline A\in C_2^{\rm ph}\setminus C_2^{\rm pt}$ from above.

\begin{center}
\begin{figure}[h!]
\begin{tikzpicture}[scale=1.5]

\draw[->](-1.3,0)--(1.3,0);\put (60,-2) {$x$};
\draw[->](0,-1.3)--(0,1.3);\put (-2,60) {$y$};
\filldraw[red, opacity=0.5]  (1,1) -- (-1,1) -- (-1,-1) -- (1,-1) --cycle;
\draw[red]  (1,1) -- (-1,1) -- (-1,-1) -- (1,-1)  --cycle;
\draw [ domain=-1:1,samples=100,blue] plot (\x, {sqrt((1-\x)/2)});
\draw [  domain=-1:1,samples=100,blue] plot (\x, {-sqrt((1-\x)/2)});
\filldraw [domain=-1:1,samples=100,blue, opacity=0.4] plot (\x, {sqrt((1-\x)/2)});
\filldraw [domain=-1:1,samples=100,blue, opacity=0.4] plot (\x, {-sqrt((1-\x)/2)});
\draw[blue] (-1,-1)--(-1,1);
\fill [blue, opacity=0.4] (-1,-1)--(1,-0.0264)--(1,0.0264)--(-1,1);
\filldraw (1,1) circle (0.8pt); \put (48,42) {$\underline A$};

\end{tikzpicture}
\caption{affine section of $C_2^{\rm ph}$ (red) and $C_2^{\rm pt}$ (blue)}\label{fig:sec}
\end{figure}
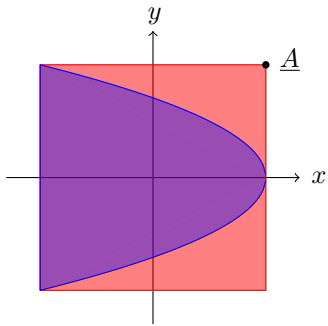
\end{center}
\end{example}

\bibliographystyle{plain}
\bibliography{references}

\end{document}